\documentclass[11pt,a4paper]{amsart}

\usepackage{lmodern}

\usepackage{todonotes}
\usepackage[english]{babel}

\usepackage[utf8]{inputenc} 	
\usepackage[T1]{fontenc}    	
\usepackage{hyperref}       	
\usepackage{amssymb}            
\usepackage{url}            		
\usepackage{booktabs}       	
\usepackage{amsfonts}       	
\usepackage{nicefrac}       	
\usepackage{microtype}      	
\usepackage{lipsum}			
\usepackage{graphicx}
\usepackage{subcaption}
\usepackage{hyperref} 				
\usepackage{tikz}
\usepackage{amsmath}
\usepackage{enumerate}
\usepackage{mwe}
\usepackage[shortlabels]{enumitem} 	
\setlist[itemize]{noitemsep} 			
\usepackage{graphbox}				
\usepackage{float}
\usepackage[toc,page]{appendix}
\usepackage{mathtools}
\usepackage{stmaryrd}
\usepackage{faktor}
\usepackage{cancel}
\usepackage[alphabetic, abbrev]{amsrefs}  

\usepackage{tikz}
\usepackage{stackengine}

\usepackage{upgreek}

\theoremstyle{plain}
\newtheorem*{definition*}{Definition}
\newtheorem{theorem}{Theorem}

\newtheorem{lemma}[theorem]{Lemma}

\newtheorem{corollary}[theorem]{Corollary}
\theoremstyle{definition}

\theoremstyle{remark}
\newtheorem*{example*}{Example}
\newtheorem{remark}[theorem]{Remark}
\newtheorem{example}[theorem]{Example}

\hypersetup{
    colorlinks,
    linkcolor={red!50!black},
    citecolor={blue!50!black},
    urlcolor={blue!80!black}
}

\DeclareMathOperator{\Lie}{Lie}

\newcommand{\R}{\mathbb{R}}			
\newcommand{\N}{\mathbb{N}}			
\newcommand{\F}{\mathcal{F}}			
\newcommand{\A}{\mathcal{A}}			
\newcommand{\Or}{\mathcal{O}}		
\newcommand{\Proj}{\mathbb{R}\mathbb{P}}			

\newcommand{\spa}{\mathrm{span}}

\newcommand{\M}{\mathcal{M}}

\newcommand{\cl}{\operatorname{cl}}

\providecommand{\keyword}[1]
{
  \small	
  \textbf{Keywords:} #1
}

\usepackage[bottom=3cm, top=3cm, left=3.5cm, right=3.5cm]{geometry}

\newcommand{\mysetminusD}{\hbox{\tikz{\draw[line width=0.6pt, line cap=round] (3pt,0) -- (0,6pt);}}}
\newcommand{\mysetminusT}{\mysetminusD}
\newcommand{\mysetminusS}{\hbox{\tikz{\draw[line width=0.45pt, line cap=round] (2pt,0) -- (0,4pt);}}}
\newcommand{\mysetminusSS}{\hbox{\tikz{\draw[line width=0.4pt, line cap=round] (1.5pt,0) -- (0,3pt);}}}
\newcommand{\mysetminus}{\mathbin{\mathchoice{\mysetminusD}{\mysetminusT}{\mysetminusS}{\mysetminusSS}}}

\title[Approximately controllable bilinear systems are controllable]{Approximately controllable finite-dimensional\\bilinear systems are controllable}

\date{\today} 					

\author{Daniele Cannarsa}
\address{Daniele Cannarsa, Universit\'e de Paris, Sorbonne Universit\'e, CNRS, Inria, Institut de Math\'ematiques de Jussieu-Paris Rive Gauche, F-75013 Paris, France}
\email{daniele.cannarsa@imj-prg.fr}

\author{Mario Sigalotti}
\address{Mario Sigalotti, Laboratoire Jacques-Louis Lions, Sorbonne Université, Université de Paris, CNRS, Inria, Paris, France}
\email{mario.sigalotti@inria.fr}

\begin{document}

\maketitle

\begin{abstract}
We show that a bilinear control system  is approximately controllable if and only if it is controllable in $\R^{n}\mysetminus\{0\}$. We approach  this {property} 
by looking at the foliation made by the orbits of the system, and by showing that  there does not exist a codimension one foliation in $\R^{n}\mysetminus\{0\}$ with dense leaves that are everywhere transversal to the radial direction.
The proposed geometric approach allows to extend the results to homogeneous systems that are angularly controllable.
\end{abstract}
\bigskip\bigskip
 \keyword{{bilinear control systems, controllability, foliations}}
 \bigskip
\setcounter{tocdepth}{1}


\section{Introduction and statement of the main result}

In this note we study  control systems of the form
	\begin{equation}\tag{$\Sigma$} \label{eqn:control}
	 \dot x = M(t)x,
	\end{equation}
where  the state $x$ is in $\R^{n}\mysetminus \{0\}$, $n\geq 1$, and 
the admissible controls 
{
$M:[0,+\infty)\to \mathcal{M}$ 
take
values in a subset $\mathcal{M}$ of the space 
$M_{n}(\R)$ of  $n\times n$ matrices with real coefficients.
For notational simplicity we shall take as admissible controls
piecewise constant functions with values in $\mathcal{M}$.}
By a slight abuse of notation, we refer to control systems such as \eqref{eqn:control}
as \emph{bilinear control systems}, although the latter term usually 
denotes systems for which 
$M(t)=A+u^{1}(t)B_{1}+\dots +u^{m}(t)B_{m}$ for some fixed $A, B_{1}, \dots, B_{m}\in M_n(\R)$, with control $t\mapsto (u^1(t),\dots,u^{m}(t))$ taking values in some subset $U$ of $\R^{m}$. 
 %
For an introduction to bilinear control system we refer to \cite{Colonius00} and \cite{elliott2009bilinear}.
{Notice that bilinear hybrid systems of the form $\dot x=(A_{q(t)}+u^{1}(t)B_{1,q(t)}+\dots +u^{m}(t)B_{m,q(t)})x(t)$, where $q(t)$ is a control parameter with discrete values (cf. \cite{Petreczky-vanSchuppen}) are also included in our setting.}

Given  $t>0$,  an initial state $x_{0}\in \R^{n}$, and an admissible control $M:[0,+\infty)\to \mathcal{M}$, let us denote by $x(t,x_{0},M)$ the value at time $t$ of the solution of \eqref{eqn:control} with initial state $x_{0}$ and control $M$.
The attainable set from an $x_{0}$ in $\R^{n}$ is
	\[
	\A_{x_{0}} = \{x(t,x_{0},M) \mid  t\geq 0, ~ M:[0,+\infty)\to \mathcal{M}\mbox{ piecewise constant}  \}.
	\]
System \eqref{eqn:control} is said to be  \textit{controllable} if,  for all state $x$, one has $ \A_{x}=\R^{n}\mysetminus\{0\}$,
	while it is said to be \textit{approximately controllable} if, for all  state $x$, one has
	$
	\cl(\A_{x})=\R^{n}$,
where $\cl$ denotes the closure of a set in $\R^{n}$.
The main result of this note is the following. 

\begin{theorem}\label{thm:conjecture-1}
	\eqref{eqn:control} is approximately controllable if and only if  \eqref{eqn:control} is controllable.
	\end{theorem}

This result shows that the a priori weaker notion of approximate controllability implies controllability with no additional assumption, other than the 
finite-dimensional 
system being bilinear.
This is useful when topological arguments lead directly to approximate controllability results, as it is the case for constructions based on control sets, whose definition involve the closure of attainable sets \cite{Colonius00}. 
For example, in \cite{colonius2021control} the authors use Theorem~\ref{thm:conjecture-1} to conclude that their sufficient condition for approximate controllability,  expressed in terms of
 the Floquet spectrum of the bilinear system, 
actually yields controllability (see  \cite[Corollary~3.19]{colonius2021control}).

Another possible way of applying Theorem~\ref{thm:conjecture-1} is the following: if for a bilinear system one is able to identify   vector fields 
compatible with \eqref{eqn:control} in the sense of \cite[Definition~8.4]{book} that, when added to the admissible ones, lead to approximate controllability, then controllability of \eqref{eqn:control} follows, without the need of checking the Lie algebra rank condition.
Such an extension argument by compatible vector fields is, e.g.,  at the core of the results in \cite{1412010}, which can therefore be improved by our result. In particular, Theorem~\ref{thm:conjecture-1} implies that the hypotheses ii)  and iii) on the existence of stable and antistable equilibria can be dropped from   
\cite[Theorem~4.3]{1412010}. Similarly, Propositions~3.3 and 3.6
from the same paper 
can be strengthened by replacing in their conclusions approximate and practical controllability by controllability. 

The result in Theorem~\ref{thm:conjecture-1} is in sharp contrast with the case of bilinear systems in infinite-dimension: when the controlled operators $B_i$ appearing in the representation $M(t)=A+u^1(t)B_1+\dots+u^m(t)B_m$ are bounded,
these systems cannot be controllable (see \cite[Theorem~3.6]{BMS82} and also \cite{BoussaidCaponigroChambrion-JFA} for recent extensions), while there exist some criteria for approximate controllability (see, e.g., \cite[Chapters 4 and 9]{Khapalov}, \cite{CannarsaFloridiaKhapalov} and \cite{BoscainCaponigroSigalotti}). 

 For general finite-dimensional systems (to which the notions of controllability and approximate controllability straightforwardly extend), while controllability clearly implies approximate controllability, the converse may fail to hold. 
A standard example  can be provided using the 
irrational winding of a line in the torus {$\mathbb{T}^{n}$, $n\ge 2$}. 
%
%
%
%

On the other hand, the equivalence stated in Theorem~\ref{thm:conjecture-1} is known to hold for 
some classes of control systems. 
This is the case for 
linear control systems in $\R^{n}$, i.e., systems of the form
\begin{equation}\label{eq:lin_cont_sys}
  \dot x = A x + B u(t), \qquad 
  u:[0,+\infty)\to \R^{m},\quad x\in \R^{n},
\end{equation}
{with} 
$A\in M_{n}(\R)$ and $B\in M_{n\times m}(\R)$.
Indeed, if system  {\eqref{eq:lin_cont_sys}} is approximately controllable, then the attainable set from the origin $\A_{0}$ is dense in $\R^n$. Since  $\A_{0}$ is a linear space (and in particular it is closed), it follows that $\A_{0}=\R^n$, which is well known to be equivalent to the controllability of
{\eqref{eq:lin_cont_sys} due to the linear structure of the system}.

Few other classes of control systems for which approximate controllability implies controllability are known:  
closed quantum systems on 
$S^{n-1}$
 \cite[Theorem~17]{Boscain_2014};  
right-invariant control systems on {simple} Lie groups {(as it follows from \cite[Lemma~6.3]{JURDJEVIC1972313} and \cite[Note at p.~312]{smith})};    
  control systems obtained by projecting  onto $\Proj^{n-1}$ systems of the form of \eqref{eqn:control} \cite[Proposition~44]{Boarotto_2020}.

By proving Theorem~\ref{thm:conjecture-1}, we are able to identify {bilinear control systems as a} new class 
of systems for which approximate controllability and controllability are equivalent. 
The proposed proof of Theorem~\ref{thm:conjecture-1} (which can be found in Section~\ref{sec:3}), works 
first by deducing from \cite{Boarotto_2020} and \cite{article} that, if
{the projection of \eqref{eqn:control} onto $\Proj^{n-1}$}
is approximately controllable, then the orbits of \eqref{eqn:control} are 
transversal to the radial direction, and then by proving 
that there does not exist a codimension one foliation in $\R^{n}\mysetminus\{0\}$ with dense leaves which are transversal to the radial direction. 

As a byproduct of this demonstration strategy, we can extend Theorem~\ref{thm:conjecture-1} to  angularly controllable homogenous control systems (see Corollary~\ref{cor:homogeneous}).

%

\subsection*{Acknowledgements}
This work was supported by the Grant 
{ANR-17-CE40-0007-01 QUACO}
of the French ANR.
DC is supported by the DIM Math Innov grant from R\'egion \^Ile-de-France. The authors want to thank Francesco Boarotto, Ugo Boscain and Leonardo Martins Bianco for preliminary discussions on the subject of this note. 

\section{Notations and related results from the literature}\label{sec:2}

Given a matrix $M$ in $\mathcal{M}$, let us denote by $f_{M}$ the associated vector field 
	\[
	f_{M}:x\mapsto Mx, \qquad x \in \R^{n}\mysetminus\{0\}.
	\]
Since the vector fields $f_{M}$ are $\R$-homogenous, 
$\mathcal{F}^{\mathcal{M}}=\{f_{M}\mid  M\in \mathcal{M}\}$ is a family of analytical, homogenous vector fields in $\R^{n}\mysetminus\{0\}$. 
System \eqref{eqn:control} is a special case of a general control system
	\begin{equation}\label{eqn:system}\tag{C}
	\dot x = f_{u(t)}(x), 
	\end{equation}
where the state $x$ evolves in a {$n$-dimensional} manifold $X$ following a fixed family $\F$ of admissible vector fields indexed by a set $U$, i.e., $\F=\{f_u\mid u\in U\}$, with piecewise constant  controls $u:[0,+\infty)\to U$. 
(For system \eqref{eqn:control}  we have $X=\R^{n}\mysetminus \{0\}$ and $\F=\F^{\M}$.)
For notational simplicity, we assume that each  vector field $f$ in $\F$ is complete. We denote by $e^{tf}$ the flow at time $t$ of the equation $\dot x=f(x)$.
The \emph{attainable set from a point $x$ in $X$ for \eqref{eqn:system}} is
	\[
	\A_{x}	= \{e^{t_{1}f_{1}}\circ \dots \circ e^{t_{k}f_{k}}(x)\mid k\in \N,~ t_{1},\dots,t_k\geq 0, ~ f_{1},\dots,f_k\in \F \};
	\]
similarly, the \emph{orbit of $x$ for \eqref{eqn:system}} is
	\[
	\mathcal{O}_{x} 	= \{e^{t_{1}f_{1}}\circ \dots \circ e^{t_{k}f_{k}}(x)\mid k\in \N,~ t_{1},\dots,t_k\in \R,~ f_{1},\dots,f_k\in \F  \}.
	\]
The difference between  orbits and  attainable sets is that the orbits are constructed by using positive and negative time directions; thus, $\A_{x}\subset \mathcal{O}_{x}$ for all $x$ in $X$.
The orbits satisfy the following important property.

\begin{theorem}[Orbit theorem, \cite{10.2307/1996660}]\label{thm:orbit}
	For every $x\in X$, the orbit $\mathcal{O}_{x}$ is a connected, immersed submanifold of $X$.
	Moreover, 
	for all $y\in \mathcal{O}_{x}$,
	\[
	T_{y}\mathcal{O}_{x}= \spa_{\R}\big\{(e^{t_{1}f_{1}}_{*}\circ \dots \circ e^{t_{k}f_{k}}_{*}g)(y)\mid k\in \N,~ t_{1},\dots,t_k\in \R, ~ f_{1},\dots,f_k,g\in \F\big\},
	\]
	where $e^{t_{i}f_{i}}_{*}$ denotes the  operator of push-forward of vector fields along the flow $e^{t_{i}f_{i}}$. 
\end{theorem}
It follows from the orbit theorem that the orbits define a \textit{foliation}, i.e., a partition of $X$ in immersed, connected, submanifolds (called \textit{leaves}), possibly of varying dimension.  
In some texts the term foliation describes what we shall call {here} a \textit{regular foliation}, i.e., a foliation 
whose leaves have all the same dimension $d$,
and which 
admits locally around each point a coordinate trivialization of the form $\{\omega+p\mid p\in P\}$
for some connected submanifold $\omega\subset \R^n$ of dimension $d$ 
and some connected submanifold $P$ of dimension $n-d$ transversal to $\omega$. 

{Another class of foliations of interest for our discussion is that of \emph{integral}
foliations, i.e., foliations such that for every $x$ in $X$  
every vector $v\in T_x X$ tangent to the leaf passing through $x$
can be extended to a smooth vector field {everywhere} tangent to 
the leafs of the foliation. 
The fact that foliations arising as orbits of control systems as in \eqref{eqn:system} are integral }
is a consequence of the orbit theorem.  
Conversely, any {integral} 
foliation is the orbit partition of a control system.
For this, it suffices to take as family of admissible vector fields 
the collection of all vector fields that are everywhere tangent to the leaves.

\medskip

For any family $\mathcal{G}$ of vector fields, we denote by $D^{\mathcal{G}}$ the distribution defined by $\mathcal{G}${, that is,} 
	\[
	D^{\mathcal{G}}=\{D_{x}^{\mathcal{G}} \mid  x\in X\}, \qquad D_{x}^{\mathcal{G}}=\spa
	\{f(x) \mid  f\in \mathcal{G}\}\quad \forall~ x\in X.
	\]
The distribution $D^{\mathcal{G}}$ may not be a sub-bundle of $TM$.
Now, let $\Lie\mathcal{F}$ be the smallest Lie algebra of vector fields containing $\mathcal{F}$, i.e., 
	\[
	\Lie\mathcal{F}=\spa\big\{[f_{1}, \dots[f_{k-1},f_{k}]\dots] \mid  k\in\N,\ f_{1}, \dots, f_{k}\in \F\big\}.
	\] 
Since the vector $[f_{i}, f_{j}](x)$ is tangent to the trajectory $t\mapsto e^{-\sqrt{t}f_{j}}\circ e^{-\sqrt{t}f_{i}}\circ e^{\sqrt{t}f_{j}}\circ e^{\sqrt{t}f_{i}}(x)$, one has that 
	\begin{equation}\label{eq:lie-inclusion}
	D^{\Lie\F}_{y}\subset T_{y}\Or_{x}, \qquad \forall ~ x\in X, ~ y \in \Or_{x}.
	\end{equation}
Under some rather general hypothesis, the inclusion in \eqref{eq:lie-inclusion} turns out to be an equality, i.e., 
	\begin{equation}\tag{LD}\label{eqn:lie-equality} 
	D^{\Lie\F}=T\Or.
	\end{equation}
If this is the case, we say that   \eqref{eqn:system} is \emph{Lie-determined}.
For example, any of the following classical hypotheses imply that system \eqref{eqn:system} is Lie-determined:
\begin{enumerate}[(i)]
\item  the family {$
\F$} is analytic \cite{Nagano1966LinearDS}; 
\item $\Lie \F$ is locally finitely generated \cite{10.2307/24900924};
\item  $D^{\F}$ is $\F$-invariant \cite{10.2307/1996660} \cite{10.1112/plms/s3-29.4.699};
\item  $D^{\Lie \F}$ is a constant rank distribution \cite{frobenius1877ueber}. In this case the foliation described by the orbits is regular.
\end{enumerate} 
For a short introduction {to} 
the subject we refer to \cite{Lavau_2018}. 
Notice that {our bilinear system} \eqref{eqn:control} is Lie-determined since the vector fields $f_{M}$ are analytic. 
Our proof {uses} 
the following classical 
result about attainable sets. It says that, {in the Lie-determined case},
orbits cannot be of a greater dimension than attainable sets.

\begin{theorem}[Krener's theorem \cite{krener1974generalization}]
Assume that \eqref{eqn:system} is Lie-determined.
Then, for every $x$ in $X$, the attainable set $\mathcal{A}_{x}$ has a nonempty interior in the orbit $\Or_{x}$. 
\end{theorem}

The following lemma provides some informations about 
the orbits of an approximately controllable system.

\begin{lemma} \label{lem:existance-foliation}
Assume that system \eqref{eqn:system} is approximately controllable. Then, the orbits of \eqref{eqn:system} form a regular foliation of $X$ with dense leaves.
\end{lemma}

\begin{proof} 
{As noticed 
above,} 
the orbits of system \eqref{eqn:system} form a partition of $X$ in immersed submanifolds. 
Since the attainable sets are contained in the orbits, the approximate controllability implies that the orbits are dense. 
Finally, due to the expression  of the tangent space of the orbits in Theorem \ref{thm:orbit}, the dimension of the orbits is lower semi-continuous, i.e., for all $x$ in $X$, there exists a neighbourhood $V(x)$ of $x$ such that 
	\[
	\dim \Or_{x}\leq \dim \Or_{y},\qquad \forall ~ y \in V(x).
	\]
Now let $\Or_{x}$ be {an} 
orbit of maximal dimension; since $\Or_{x}$ is dense, all  other orbits have the same dimension as $\Or_{x}$.
Finally, an {integral}
foliation of constant rank is a regular foliation.
\end{proof}

The family $\F$ is said to satisfy the \emph{Lie algebra rank condition at $x\in X$} 
 if the evaluation  at $x$ of the Lie algebra generated by $\F$  has maximal dimension, i.e., $ D^{\Lie\mathcal{F}}_{x}=T_x X$.

\begin{corollary}\label{cor:existance-foli}
Assume that {system \eqref{eqn:system} is 
Lie-determined and approximately controllable}.
Then, exactly one of the following holds:
\begin{enumerate}[\it (a)]
\item\label{it:cont} $\F$ satisfies the Lie algebra rank condition at all points in $X$; thus, system \eqref{eqn:system} is controllable.
\item\label{it:foli} There exists an integer $k$ with $0<k<\dim X$ such that 
the orbits of \eqref{eqn:system} form a regular $k$-dimensional foliation of $X$ with dense leaves.
\end{enumerate}
\end{corollary}

\begin{proof}
If there exists a point $x$ in $X$ at which $\F$ satisfies  the Lie algebra rank condition, then \eqref{eqn:system} has a single
orbit. Indeed, Krener's theorem implies that the interior of $\A_{x}$ is nonempty. Hence,  due to the approximate controllability assumption, the attainable set from any other point intersects $\A_{x}$, and therefore is contained in $\Or_x$. 
Thus, \eqref{eqn:system} has a single orbit and, due to 
{the Lie-determinedness property} 
\eqref{eqn:lie-equality}, every point in $X$ satisfies the Lie algebra rank condition. The controllability  follows 
{as a corollary of
Krener's theorem 
(}\cite[Corollary~8.3]{book}).

Otherwise, assume that $\F$ does not satisfy the Lie algebra rank condition at any point. Then, Lemma~\ref{lem:existance-foliation} shows that the orbits of \eqref{eqn:system} form a regular foliation, whose leaves are dense since they contain the attainable sets. 
Finally,  the dimension of the orbits is less than the dimension of $X$, otherwise $\F$ would satisfy the Lie algebra rank condition {due to \eqref{eqn:lie-equality}}.
\end{proof}

\begin{remark}
Corollary \ref{cor:existance-foli} is useful if one can exclude the existence of a foliation with the properties described in \ref{it:foli}; this might be possible thanks to 
the particular form of system \eqref{eqn:system} or 
some topological properties of $X$. Most of the results in this direction are for codimension-one {regular} foliations: for example, it is known that even dimensional spheres do not admit codimension-one {regular}  foliations \cite{10.2307/1970795}. We recall also that a compact manifold with finite fundamental group has no analytic {regular} foliations of codimension one \cite{Haefliger1957/58}; some additional results can be found in \cite{lawson74}.
\end{remark}

If we drop the hypothesis that the system is Lie-determined, it is possible to construct an approximately controllable and not controllable system having only one orbit.
This is shown in the following example.
\begin{example}\label{ex:non-c}
Let $X=\R^{2}$ and consider the family $\F=\{f_{1},~f_{2},~f_{3}^{+},~f_{3}^{-}\}$ with
	\[
	f_{1}=\frac{\partial}{\partial y}, \quad f_{2}=-\phi(x,y)f_{1}, \quad f_{3}^{\pm}=\pm \phi(x,y)\frac{\partial}{\partial x},
	\] 
where $\phi:\R^{2}\to [0,+\infty)$ is a smooth function  such that, for all $(x,y)\in \R^{2}$, $\phi(x,y)=0$ if and only if  $x=0$ and $y\leq 0$. 
This four vector fields are illustrated in  Figure~\ref{img:figure-example}.
It is not hard to {check} 
that this system has a single orbit, is approximately controllable, {and still it }
is not controllable. 
\begin{figure}[]
\begin{center}
\includegraphics[width=0.45\linewidth]{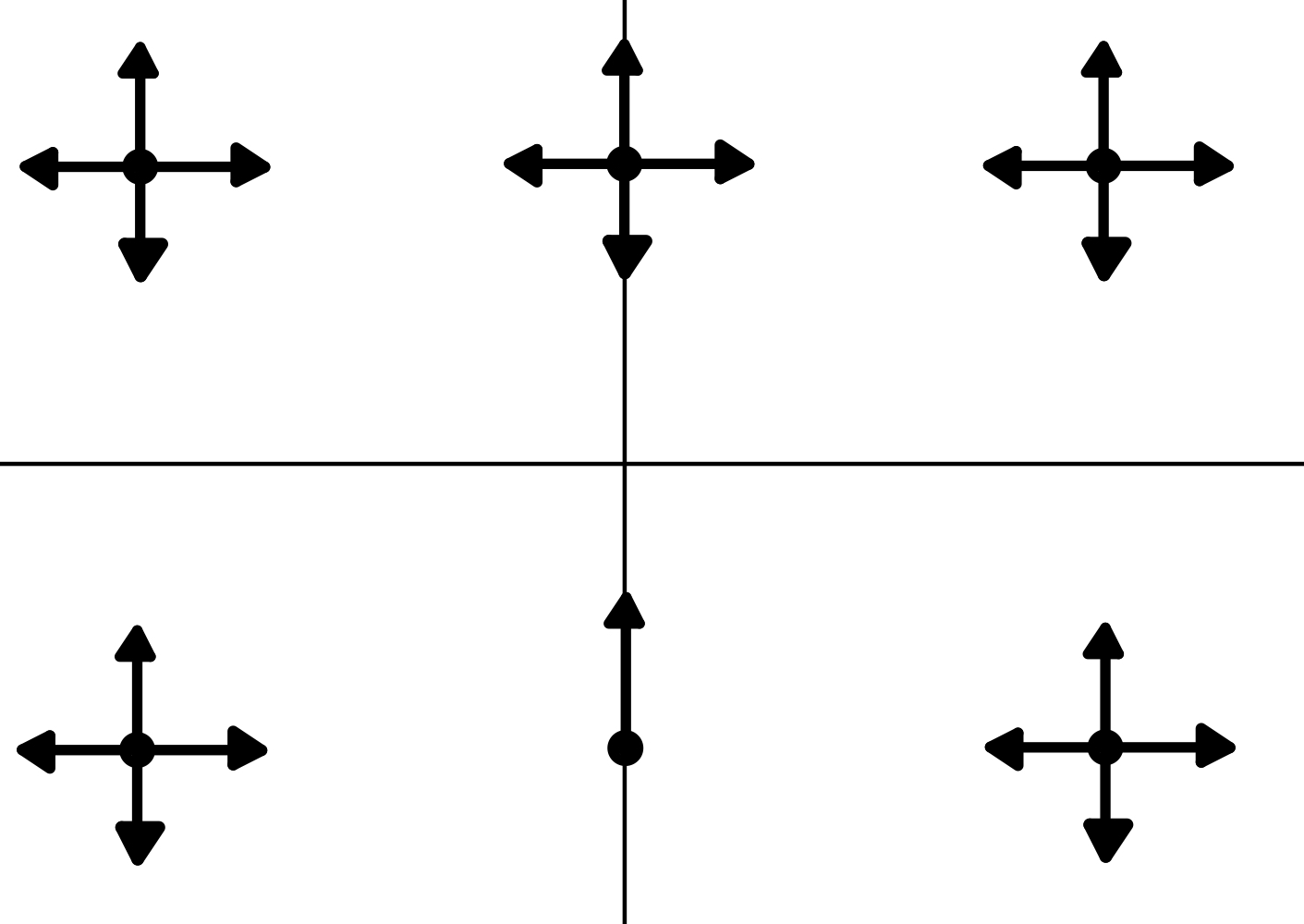}
\caption{\label{img:figure-example}\small The admissible vector fields 
for the control system 
introduced in 
Example \ref{ex:non-c}.}
\end{center}
\end{figure}
\end{example}

\section{Proof of Theorem \ref{thm:conjecture-1}}\label{sec:3}

Assume that system \eqref{eqn:control} is approximately controllable and that $n\geq 2$, the case $n=1$ being trivial.
 As already remarked, 
system~\eqref{eqn:control} is Lie-determined due to the analyticity of each linear vector field. Thus, Corollary \ref{cor:existance-foli} applies: either system \eqref{eqn:control} is controllable, or the partition of $\R^{n}\mysetminus\{0\}$ into the orbits of \eqref{eqn:control} forms a regular foliation of positive codimension as described in \ref{it:foli}. Let us assume the latter, and show that this leads to a contradiction.

Let us introduce two systems which can be naturally {associated with} 
 \eqref{eqn:control}: the projections of system  \eqref{eqn:control} onto $S^{n-1}$ and 
$ \Proj^{n-1}$.
First, consider the projection $\pi:\R^{n}\mysetminus \{0\}\to S^{n-1}$ defined by $\pi(x)=x/|x|$.
For every $x$ in $S^{n-1}
$ {and every $v\in T_{x}\R^{n}$}, consider the pushforward $\pi_{*}(v)= v -  \langle x,v\rangle x$,
and let us introduce the
control system
	\begin{equation} \tag{S$\Sigma$} \label{eqn:control-sphere}
  	\dot x = \pi_{*}(M(t) x), \qquad  
	M(t)\in \mathcal{M},  ~ x\in  S^{n-1}.
	\end{equation}
Due to homogeneity, the  trajectories of \eqref{eqn:control-sphere} are the image of the trajectories of \eqref{eqn:control} via $\pi$; thus, the orbits $\Or^{S}$ of \eqref{eqn:control-sphere} are projections of the orbits of \eqref{eqn:control}, that is 
 	\begin{equation}\label{eq:orbit-projection}
	\Or_{\pi (y)}^{S} = \pi \big(\Or_{y}\big), \qquad \forall y \in \R^{n}\mysetminus \{0\}.\end{equation}
We say that \eqref{eqn:control} is \emph{angularly controllable} if \eqref{eqn:control-sphere} is controllable.
Similarly, consider the canonical projection $\omega:\R^{n}\mysetminus\{0\}\to \Proj^{n-1}$ and the system
\begin{equation}
  \tag{$\mathbb{P}\Sigma$} \label{eqn:control-proj}
  \dot q = \omega_{*}( M(t)x), \qquad   
  M(t)\in \mathcal{M},  ~  q=\omega(x)\in \Proj^{n-1}.
\end{equation}
This system is well-defined because $\omega_{*}( M(t)x)$ depends only on $q$ and not on the
choice of the specific $x\in\R^n\mysetminus\{0\}$ such that $q=\omega(x)$.
As mentioned in the introduction, Theorem~\ref{thm:conjecture-1} holds if one replaces  \eqref{eqn:control} by \eqref{eqn:control-proj}, as shown in \cite[Proposition~44]{Boarotto_2020}. 
Using this result, let us deduce the following property. 

\begin{lemma}\label{lem:implies-angul-control}
	If \eqref{eqn:control} is approximately controllable then \eqref{eqn:control} is angularly controllable.
\end{lemma}
\begin{proof}
Since the projection of the trajectories of \eqref{eqn:control} are trajectories of \eqref{eqn:control-proj}, if the former is approximately controllable then the same holds for the latter. 
Due to \cite[Proposition~44]{Boarotto_2020}, if \eqref{eqn:control-proj} is approximately controllable then it is controllable. 
In \cite[Theorem~1]{article} the authors show that system \eqref{eqn:control-proj} is controllable if and only if system \eqref{eqn:control-sphere} is controllable. 
Therefore, system \eqref{eqn:control-sphere}  is controllable, meaning that system \eqref{eqn:control} is angularly controllable.
\end{proof}

Let us denote by $\Or = \{\Or_{x} \mid x \in \R^{n}\mysetminus\{0\}\}$ the orbit partition of system \eqref{eqn:control}.  Due to \eqref{eq:orbit-projection} and  Lemma~\ref{lem:implies-angul-control}, one has
	\[
	\pi_{*} \big(T_{y}\Or_{x}\big) = T_{\pi(y)}\Or^{S}_{\pi(x)}=T_{\pi(y)}S^{n-1}, 
	\qquad \forall ~x \in \R^{n}\mysetminus\{0\}, ~\forall~y \in \Or_{x}.
	\]
Therefore, we have that
	\begin{equation} \label{eqn:transverse-lie}
	T_{y}\Or_{x}+\R y = \R^{n}, \qquad \forall~ x \in \R^{n}\mysetminus\{0\}, ~\forall~y \in \Or_{x},
	\end{equation}
which is to say that the orbits are transversal to the radial direction. 
Additionally, since we assumed to be in case \ref{it:foli} of Corollary \ref{cor:existance-foli}, this implies that 
$\dim T_{y}\Or_{x}= n-1$. 
It follows that $\Or$ is a codimension-one regular foliation of $\R^{n}\mysetminus\{0\}$ transversal to the radial direction and with dense leaves.
{Moreover $\Or$ is \textit{homogenous} in the sense that, for all $\lambda>0$ and $x$ in $\R^{n}\mysetminus\{0\}$, one has that $\lambda\Or_{x}=\Or_{\lambda x}$.}
In the following lemma we show that such a foliation cannot exist.

\begin{lemma} \label{lem:existence-larc} 
	Assume that $n\geq 2$.
	Then, there does not exist a homogenous, codimension-one regular foliation of $\R^{n}\mysetminus\{0\}$ transversal to the radial direction and with dense leaves.
\end{lemma}

The hypothesis  of transversality between the foliation and the radial direction is necessary, as counterexamples can be constructed otherwise. 
For instance, in \cite{AIF_1976__26_1_239_0} the author presents an example of codimension-one regular foliation of $\R^{3}$ with dense leaves; this construction is presented with additional details in  \cite[Chapter~4]{camacho2013geometric}.

\begin{proof}[Proof of Lemma~\ref{lem:existence-larc}] 
	By contradiction, suppose  there exists a codimension-one regular  foliation $\mathcal{L}=\{L_{\alpha}\mid \alpha \in A\}$ 
	 of $\R^{n}\mysetminus\{0\}$ with dense leaves transversal to the radial direction.
	
Let us first consider the case $n= 2$. 
Orienting the foliation using the clockwise direction and applying Whitney's theorem (see  \cite[Theorem~2.3 at p.~23]{ABZintrod}), the foliation can be identified with the set of trajectories
of a vector field. 
Using the stereographic projection, the flow of such a vector field can be pushed to the sphere $S^{2}$ minus two points.
However, a flow with dense trajectories on $S^{2}$ minus finitely many points
 does not exist: see, for instance, \cite[Lemma~2.4 at p.~56]{ABZintrod}.

\begin{figure}[]
\begin{tikzpicture}
    \draw (0, 0) node[inner sep=0] {\includegraphics[width=0.35\linewidth]{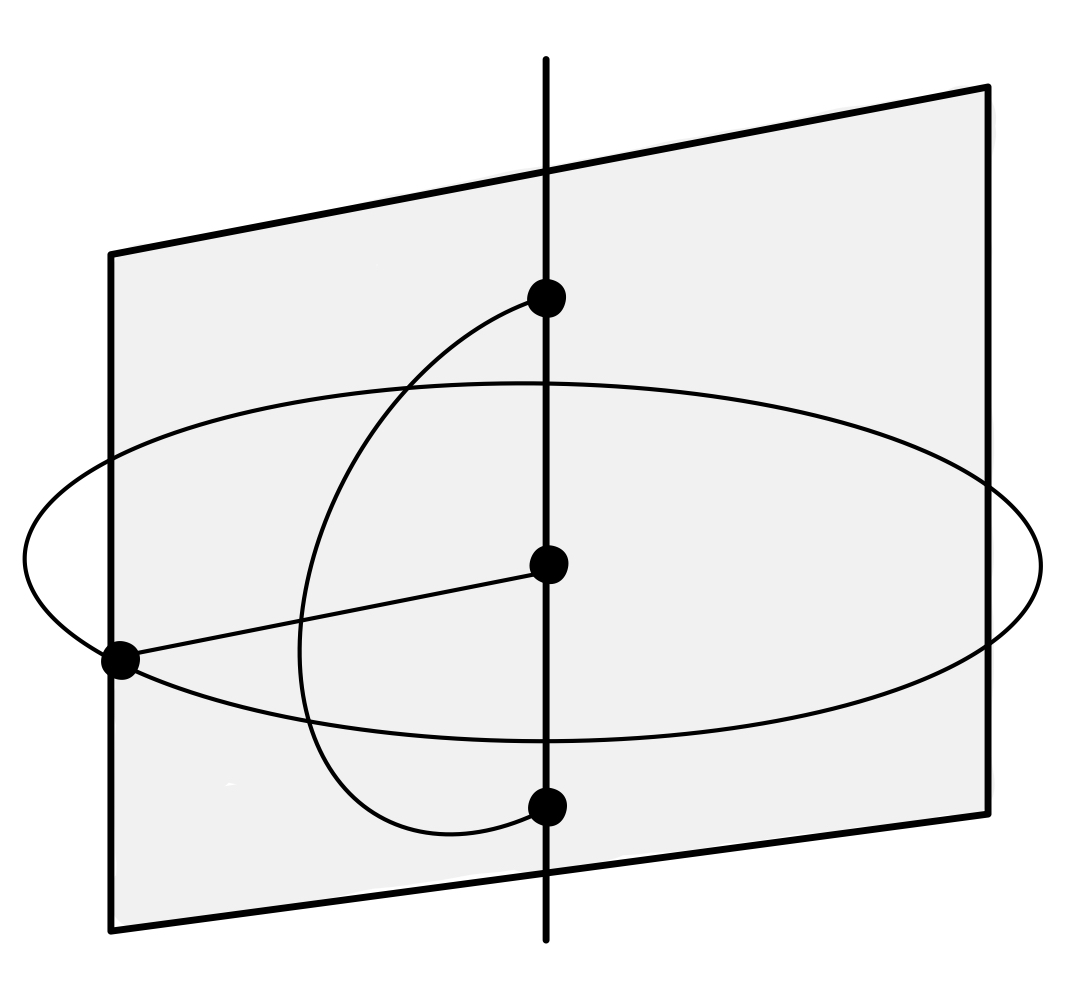}};
    \draw (0.4, 1) node {$p$};
    \draw (-2.2,-1) node {$\theta$};
    \draw (-1.2,0) node {$C_{\theta}$};
    \draw (0.4, -1.4) node {$p_{\theta}$};
    \draw (1.8, 1.5) node {$\mathcal{P}_{\theta}$};
    \draw (2.8, -0.5) node {$S^{n-2}$};
\end{tikzpicture} 
	\caption{\label{img:proof-lie-max-gen}}
\end{figure}

	Assume that $n\geq 3$. Let us fix the point $p= (0, \dots, 0, 1)\in \R^{n}$, and denote  by $S^{n-2}$ the embedded sphere $S^{n-2}\times \{0\}\subset\R^{n}$. For every $\theta$ in $S^{n-2}$, let $\mathcal{P}_{\theta}$ be the plane 
	\[\mathcal{P}_{\theta} = \spa\{p,\theta\},\]
	as depicted in Figure~\ref{img:proof-lie-max-gen}.
Because of the transversality between the leaves of $\mathcal{L}$ and 
the radial direction, the linear subbundle $I_{\theta}= \mathcal{P}_{\theta} \cap T\mathcal{L}|_{\mathcal{P}_{\theta}\mysetminus\{0\}}$ is a 1-dimensional distribution on $\mathcal{P}_{\theta}\mysetminus\{0\}$ satisfying 
	\begin{equation}\label{eqn:property-foliation}
	I_{\theta}(x)\oplus \R x = T_{x}\mathcal{P}_{\theta}, \qquad \forall ~ x\in \mathcal{P}_{\theta}\mysetminus\{0\}.
	\end{equation}
By definition, the intersection $L_p\cap \mathcal{P}_{\theta}$ contains the integral curve to $I_{\theta}$ starting at $p$.
The integral curves of $I_{\theta}$ are nothing but the leaves of the 1-dimensional foliation defined by the distribution $I_{\theta}$ in $\mathcal{P}_{\theta}\mysetminus \{0\}$. 
We claim that such a foliation is orientable. Indeed, for each $x$ in $\mathcal{P}_{\theta}\mysetminus\{0\}$, we can say that a nonzero vector $v$ in $I_{\theta}(x)$ has a positive orientation if $(v,x)$ is an oriented frame of $\mathcal{P}_{\theta}$.
It follows 
from the already cited Whitney's theorem
that the foliation defined by $I_{\theta}$ is the orbit partition of $\mathcal{P}_{\theta}\mysetminus\{0\}$ by the flow of a vector field $g_{\theta}$ everywhere transversal to the radial direction. 

Because of the transversality condition and the homogeneity, the flow of $g_{\theta}$
makes a turn  around the origin, in the sense that 
its angular 
component is monotone 
with respect to time 
and the radial 
component
does not 
converge to zero nor diverge in finite time. 
Let us choose the vector field $g_{\theta}$ such that, starting from $p$, one intersects $\R_{>0}\theta$ before $-\R_{>0}\theta$.
Define $p_{\theta}$ to be the point of first intersection between the integral curve of $g_{\theta}$ starting at $p$ and the ray $-\R_{>0}p$, and $C_{\theta}$ to be the arc between $p$ and $p_{\theta}$ (see Figure \ref{img:proof-lie-max-gen}).

Now, let us define the map $\Phi:S^{n-2}\to -\R_{>0}p$ by $\Phi(\theta) = p_{\theta}$.  The map $\Phi$ is well-defined, in the sense that $p_{\theta}$ does not depend on the vector field $g_{\theta}$ (once the latter is chosen with the appropriate orientation). 
Moreover, the map $\Phi$ is continuous, {as} it  follows from the transversality between $I_{\theta}(-p)$ and $-\R_{>0}p$ for all $\theta$ in $S^{n-2}$.
In addition, the image of $\Phi$ is contained in the intersection $L_{p}\cap -\R_{>0}p$, which has empty interior because of the transversality between $T\mathcal{L}$ and the radial direction. 
Since $S^{n-2}$ is connected ($n>2$), it follows that $\Phi$ is constant. 
 Let us define 
 \[S=\bigcup_{\theta\in S^{n-2}}C_{\theta}.\]
 {By the transversality between $I_{\theta}(-p)$ and $-\R_{>0}p$ for all $\theta$ in $S^{n-2}$ it follows that we can parameterize $C_\theta$ as a continuous arc on $[0,1]$  continuously with respect to $\theta$. 
 Hence, the topology of $S$ as a subset of $\R^{n}$ is that of $(S^{n-2} \times [0,1])/ \sim$, where $\sim$ is the equivalence relation  which identifies  the points in $S^{n-2}\times\{0\}$ to a single equivalence class, and analogously for the points in $S^{n-2}\times\{1\}$.
That is, $S$ is 
 } 
 a topological manifold homeomorphic to the sphere $S^{n-1}$. 
In particular, $S\subset L_{p}$ is closed in $\R^{n}$, and therefore it is closed in $L_{p}$.
 Since $S$ has the same topological dimension as $L_p$, we have that $S$ is open in the topology of $L_{p}$. 
 Since $L_p$ is connected, we can conclude that $S= L_{p}$. This is {contradicts} 
 the assumptions that the leaves are dense. 
  \end{proof}

 Lemma \ref{lem:existence-larc}  shows that the supposition that we are in case \ref{it:foli} of Corollary \ref{cor:existance-foli} leads to a contradiction. Therefore, we are in case \ref{it:cont}, and \eqref{eqn:control} is controllable. This concludes the proof of Theorem~\ref{thm:conjecture-1}. 
 \hfill $\square$
\\

The result in Lemma~\ref{lem:existence-larc} implies that Theorem \ref{thm:conjecture-1} generalises for control systems 
which are Lie-determined, homogeneous, and angularly controllable. 
{We say that \eqref{eqn:system} is \emph{homogeneous} if $X=\R^n\setminus\{0\}$ and for every $x\in X$, $u\in U$, and $\lambda>0$,  
$\pi_{*}(f_{u}(\lambda x))=\pi_{*}(f_{u}(x))$, where $\pi:\R^{n}\mysetminus \{0\}\to S^{n-1}$ still denotes the canonical projection. 
Just as in the bilinear case,} 
the projection of such systems on the sphere $S^{n-1}$ is
	\begin{equation} \tag{SC} \label{eqn:control-general-sphere}
  	\dot x = \pi_{*}(f_{u(t)}(x)), \qquad  
	f_{u(t)}\in \mathcal{F},  ~ x\in  S^{n-1},
	\end{equation}
and we say that \eqref{eqn:system} is angularly controllable if \eqref{eqn:control-general-sphere} is controllable.

\begin{corollary} \label{cor:homogeneous}
Let $n\geq 2$.
Assume that the  control system  \eqref{eqn:system} 
is Lie-determined, homogenous, and angularly controllable. 
Then, \eqref{eqn:system} is approximately controllable if and only if it is controllable.
\end{corollary}

Since for $n=2$  system \eqref{eqn:system} projects to $S^{1}$, in this case the hypothesis of angular controllability can be easily removed.

\begin{remark}
For all $n$ odd, $n=2k+1$ with $ k\geq 1$, the hypothesis of angular controllability in Corollary \ref{cor:homogeneous}  is superfluous. Indeed, if system \eqref{eqn:system} is Lie-determined, homogenous, and approximately controllable, then Corollary~\ref{cor:existance-foli} implies that either \eqref{eqn:system}  is angularly controllable, or the projection of its orbits forms a nontrivial regular foliation of the even-dimensional sphere $S^{2k}$.  
Since  even dimensional spheres do not admit  nontrivial {regular} foliations (indeed their tangent spaces do not admit any nontrivial subbundles - see, e.g., 
\cite[Problem 9C
]{milnor1974characteristic}), the angular controllability follows. 
\end{remark}

However, it has not been possible to fully remove the hypothesis of angular controllability in Corollary~\ref{cor:homogeneous}. 
In this regard, let us discuss the case $n=4$. 
Due to Corollary~\ref{cor:existance-foli}, if \eqref{eqn:system} is Lie-dermined, homogeneous, approximately controllable, and \eqref{eqn:control-general-sphere} is not controllable, then  the orbits of \eqref{eqn:control-general-sphere} form a {regular} foliation of $S^{3}$ of either dimension one or  codimension one. 
The latter gives a contradiction, since the Novikov compact leaf theorem implies that any codimension-one {regular} foliation of the sphere $S^{3}$ has a compact leaf \cite{novikov1965topology}. 
The former implies that the orbits of \eqref{eqn:control-general-sphere} are given by the flow of a  minimal vector field, i.e., a vector field whose orbits are dense. 
The existence of such flows has been {raised as an open question} 
in \cite{bams/1183522841} for compact metric spaces, and for the sphere $S^{3}$ has been mentioned by Smale in \cite{Smale:1998wx} under the name Gottschalk conjecture; further details can be found in \cite{FBASENER2007453}.
{If the Gottschalk conjecture were to be true, it would imply the existence of 
Lie-determined, homogenous, approximate controllable, yet not controllable systems, showing that Corollary~\ref{cor:existance-foli} fails to hold if we remove the angular controllability hypothesis. }

\bibliographystyle{alphaabbrv}
\bibliography{biblio-v5}

\begin{bibdiv}
\begin{biblist}

\bib{ABZintrod}{book}{
      author={Aranson, S.~Kh.},
      author={Belitsky, G.~R.},
      author={Zhuzhoma, E.V.},
       title={Introduction to the qualitative theory of dynamical systems on
  surfaces},
      series={Translations of Mathematical Monographs},
   publisher={American Mathematical Society},
        date={1996},
      volume={153},
        ISBN={9780821897690},
}

\bib{book}{book}{
      author={Agrachev, Andrei},
      author={Sachkov, Yuri},
       title={Control theory from the geometric viewpoint},
   publisher={Springer, Berlin, Heidelberg},
        date={2002},
        ISBN={978-3-642-05907-0},
}

\bib{BoussaidCaponigroChambrion-JFA}{article}{
      author={Boussa\"{\i}d, Nabile},
      author={Caponigro, Marco},
      author={Chambrion, Thomas},
       title={Regular propagators of bilinear quantum systems},
        date={2020},
        ISSN={0022-1236},
     journal={J. Funct. Anal.},
      volume={278},
      number={6},
       pages={108412, 66},
         url={https://doi.org/10.1016/j.jfa.2019.108412},
      review={\MR{4054112}},
}

\bib{BoscainCaponigroSigalotti}{article}{
      author={Boscain, Ugo},
      author={Caponigro, Marco},
      author={Sigalotti, Mario},
       title={Multi-input {S}chr\"{o}dinger equation: controllability,
  tracking, and application to the quantum angular momentum},
        date={2014},
        ISSN={0022-0396},
     journal={J. Differential Equations},
      volume={256},
      number={11},
       pages={3524\ndash 3551},
         url={https://doi.org/10.1016/j.jde.2014.02.004},
      review={\MR{3186837}},
}

\bib{Boscain_2014}{article}{
      author={Boscain, Ugo},
      author={Gauthier, Jean-Paul},
      author={Rossi, Francesco},
      author={Sigalotti, Mario},
       title={Approximate controllability, exact controllability, and conical
  eigenvalue intersections for quantum mechanical systems},
        date={2015},
        ISSN={0010-3616},
     journal={Comm. Math. Phys.},
      volume={333},
      number={3},
       pages={1225\ndash 1239},
         url={https://doi.org/10.1007/s00220-014-2195-6},
}

\bib{BMS82}{article}{
      author={Ball, John},
      author={Marsden, J.},
      author={Slemrod, M.},
       title={Controllability for distributed bilinear systems},
        date={1982},
     journal={SIAM Journal on Control and Optimization},
      volume={20},
}

\bib{Boarotto_2020}{article}{
      author={Boarotto, Francesco},
      author={Sigalotti, Mario},
       title={Dwell-time control sets and applications to the stability
  analysis of linear switched systems},
        date={2020},
        ISSN={0022-0396},
     journal={J. Differential Equations},
      volume={268},
      number={4},
       pages={1345\ndash 1378},
         url={https://doi.org/10.1016/j.jde.2019.08.049},
}

\bib{article}{article}{
      author={Bacciotti, Andrea},
      author={Vivalda, Jean-Claude},
       title={On radial and directional controllability of bilinear systems},
        date={2013},
     journal={Systems $\&$ Control Letters},
      volume={62},
       pages={575\ndash 580},
}

\bib{CannarsaFloridiaKhapalov}{article}{
      author={Cannarsa, Piermarco},
      author={Floridia, Giuseppe},
      author={Khapalov, Alexander~Y.},
       title={Multiplicative controllability for semilinear reaction-diffusion
  equations with finitely many changes of sign},
        date={2017},
        ISSN={0021-7824},
     journal={J. Math. Pures Appl. (9)},
      volume={108},
      number={4},
       pages={425\ndash 458},
         url={https://doi.org/10.1016/j.matpur.2017.07.002},
      review={\MR{3698164}},
}

\bib{1412010}{article}{
      author={Cheng, Daizhan},
       title={Controllability of switched bilinear systems},
        date={2005},
     journal={IEEE Transactions on Automatic Control},
      volume={50},
      number={4},
       pages={511\ndash 515},
}

\bib{Colonius00}{book}{
      author={Colonius, Fritz},
      author={Kliemann, Wolfgang},
       title={The dynamics of control},
      series={Systems \& Control: Foundations \& Applications},
   publisher={Birkh\"{a}user Boston, Inc., Boston, MA},
        date={2000},
        ISBN={0-8176-3683-8},
         url={https://doi.org/10.1007/978-1-4612-1350-5},
}

\bib{camacho2013geometric}{book}{
      author={Camacho, C.},
      author={Neto, A.L.},
       title={Geometric theory of foliations},
   publisher={Birkh{\"a}user Boston},
        date={2013},
        ISBN={9781461252924},
         url={https://books.google.fr/books?id=UZXaBwAAQBAJ},
}

\bib{colonius2021control}{misc}{
      author={Colonius, Fritz},
      author={Raupp, Juliana},
      author={Santana, Alexandre~J.},
       title={Control sets for bilinear and affine systems},
        date={2021},
}

\bib{10.2307/1970795}{article}{
      author={Durfee, Alan~H.},
       title={Foliations of odd-dimensional spheres},
        date={1972},
        ISSN={0003486X},
     journal={Annals of Mathematics},
      volume={96},
      number={2},
       pages={407\ndash 411},
         url={http://www.jstor.org/stable/1970795},
}

\bib{elliott2009bilinear}{book}{
      author={Elliott, D.},
       title={Bilinear control systems: Matrices in action},
      series={Applied Mathematical Sciences},
   publisher={Springer Netherlands},
        date={2009},
        ISBN={9781402096136},
         url={https://books.google.fr/books?id=dH6QpmVoMvkC},
}

\bib{FBASENER2007453}{incollection}{
      author={{F. Basener}, William},
      author={Parwani, Kamlesh},
      author={Wiandt, Tamas},
       title={Minimal flows},
        date={2007},
   booktitle={{O}pen {P}roblems in {T}opology~{II}},
      editor={Pearl, Elliott},
   publisher={Elsevier},
     address={Amsterdam},
       pages={453\ndash 462},
  url={https://www.sciencedirect.com/science/article/pii/B9780444522085500454},
}

\bib{frobenius1877ueber}{article}{
      author={Frobenius, Georg},
       title={Ueber das {P}faffsche {P}roblem},
        date={1877},
     journal={Journal f{\"u}r die reine und angewandte Mathematik},
      volume={1877},
      number={82},
       pages={230\ndash 315},
}

\bib{bams/1183522841}{article}{
      author={Gottschalk, W.~H.},
       title={{Minimal sets: An introduction to topological dynamics}},
        date={1958},
     journal={Bulletin of the American Mathematical Society},
      volume={64},
      number={6},
       pages={336 \ndash  351},
         url={https://doi.org/},
}

\bib{Haefliger1957/58}{article}{
      author={Haefliger, Andr{\'e}},
       title={Structures feuillet{\'e}es et cohomologie {\`a} valeur dans un
  faisceau de groupoidess},
        date={1957/58},
     journal={Commentarii mathematici Helvetici},
      volume={32},
       pages={248\ndash 329},
         url={http://eudml.org/doc/139157},
}

\bib{AIF_1976__26_1_239_0}{article}{
      author={Hector, Gilbert},
       title={Quelques exemples de feuilletages esp\`eces rares},
        date={1976},
        ISSN={0373-0956},
     journal={Ann. Inst. Fourier (Grenoble)},
      volume={26},
      number={1},
       pages={xi, 239\ndash 264},
         url={http://www.numdam.org/item?id=AIF_1976__26_1_239_0},
}

\bib{10.2307/24900924}{article}{
      author={Hermann, Robert},
       title={The differential geometry of foliations. {II}},
        date={1962},
     journal={J. Math. Mech.},
      volume={11},
       pages={303\ndash 315},
}

\bib{JURDJEVIC1972313}{article}{
      author={Jurdjevic, Velimir},
      author={Sussmann, H{\'e}ctor~J},
       title={Control systems on {L}ie groups},
        date={1972},
        ISSN={0022-0396},
     journal={Journal of Differential Equations},
      volume={12},
      number={2},
       pages={313 \ndash  329},
  url={http://www.sciencedirect.com/science/article/pii/0022039672900356},
}

\bib{Khapalov}{book}{
      author={Khapalov, Alexander~Y.},
       title={Controllability of partial differential equations governed by
  multiplicative controls},
      series={Lecture Notes in Mathematics},
   publisher={Springer-Verlag, Berlin},
        date={2010},
      volume={1995},
        ISBN={978-3-642-12412-9},
         url={https://doi.org/10.1007/978-3-642-12413-6},
      review={\MR{2641453}},
}

\bib{krener1974generalization}{article}{
      author={Krener, Arthur~J},
       title={A generalization of {C}how's theorem and the bang-bang theorem to
  nonlinear control problems},
        date={1974},
     journal={SIAM Journal on Control},
      volume={12},
      number={1},
       pages={43\ndash 52},
}

\bib{Lavau_2018}{article}{
      author={Lavau, Sylvain},
       title={A short guide through integration theorems of generalized
  distributions},
        date={2018},
        ISSN={0926-2245},
     journal={Differential Geom. Appl.},
      volume={61},
       pages={42\ndash 58},
         url={https://doi.org/10.1016/j.difgeo.2018.07.005},
}

\bib{lawson74}{article}{
      author={Lawson, H.},
       title={Foliations},
        date={1974},
     journal={Bulletin (New Series) of the American Mathematical Society},
      volume={80},
}

\bib{milnor1974characteristic}{book}{
      author={Milnor, J.W.},
      author={Stasheff, J.D.},
       title={Characteristic classes},
      series={Annals of mathematics studies},
   publisher={Princeton University Press},
        date={1974},
        ISBN={9780691081229},
         url={https://books.google.fr/books?id=5zQ9AFk1i4EC},
}

\bib{Nagano1966LinearDS}{article}{
      author={Nagano, Tadashi},
       title={Linear differential systems with singularities and an application
  to transitive {L}ie algebras},
        date={1966},
        ISSN={0025-5645},
     journal={J. Math. Soc. Japan},
      volume={18},
       pages={398\ndash 404},
         url={https://doi.org/10.2969/jmsj/01840398},
}

\bib{novikov1965topology}{article}{
      author={Novikov, Sergei~Petrovich},
       title={Topology of foliations},
        date={1965},
     journal={Transactions of the Moscow Mathematical Society},
      volume={14},
       pages={248\ndash 278},
}

\bib{Petreczky-vanSchuppen}{article}{
      author={Petreczky, Mih\'{a}ly},
      author={van Schuppen, Jan~H.},
       title={Span-reachability and observability of bilinear hybrid systems},
        date={2010},
        ISSN={0005-1098},
     journal={Automatica J. IFAC},
      volume={46},
      number={3},
       pages={501\ndash 509},
         url={https://doi.org/10.1016/j.automatica.2010.01.008},
      review={\MR{2877100}},
}

\bib{Smale:1998wx}{article}{
      author={Smale, Steve},
       title={Mathematical problems for the next century},
        date={1998},
     journal={The Mathematical Intelligencer},
      volume={20},
      number={2},
       pages={7\ndash 15},
         url={https://doi.org/10.1007/BF03025291},
}

\bib{smith}{article}{
      author={Smith, P.~A.},
       title={Everywhere dense subgroups of {L}ie groups},
        date={1942},
        ISSN={0002-9904},
     journal={Bull. Amer. Math. Soc.},
      volume={48},
       pages={309\ndash 312},
         url={https://doi.org/10.1090/S0002-9904-1942-07665-8},
      review={\MR{6547}},
}

\bib{10.1112/plms/s3-29.4.699}{article}{
      author={Stefan, P.},
       title={Accessible sets, orbits, and foliations with singularities},
        date={1974},
        ISSN={0024-6115},
     journal={Proc. London Math. Soc. (3)},
      volume={29},
       pages={699\ndash 713},
         url={https://doi.org/10.1112/plms/s3-29.4.699},
}

\bib{10.2307/1996660}{article}{
      author={Sussmann, H{\'e}ctor~J.},
       title={Orbits of families of vector fields and integrability of
  distributions},
        date={1973},
        ISSN={00029947},
     journal={Transactions of the American Mathematical Society},
      volume={180},
       pages={171\ndash 188},
         url={http://www.jstor.org/stable/1996660},
}

\end{biblist}
\end{bibdiv}

\end{document}